\documentclass[final]{siamltex}

\usepackage{bm,amsfonts,amssymb}
\usepackage{amsmath}

\usepackage{graphics}
 \usepackage{graphicx}
\usepackage{subfigure}

\title{Distributions of Demmel and Related Condition Numbers\thanks{This work was supported by the Hong Kong Research Grants Council under grant number 616911.} }

\author{Prathapasinghe Dharmawansa\thanks{Department of Communications and Networking, Aalto School of Electrical Engineering, Otakaari 5A, Espoo 02150, Finland ({\tt prathapakd@ieee.org}).}
\and Matthew~R.~McKay\thanks{Department of Electronic and Computer Engineering, Hong Kong University of Science and Technology,
Clear Water Bay, Kowloon, Hong Kong ({\tt
eemckay@ust.hk}).}\and Yang~Chen\thanks{Department of Mathematics,
Faculty of Science and Technology, University of Macau, Av. Padre Tom\'as Pereira, Taipa, Macau China ({\tt
yangbrookchen@yahoo.co.uk, yayangchen@umac.mo}).} }
\begin{document}

\maketitle

\begin{abstract}
Consider a random matrix $\mathbf{A}\in\mathbb{C}^{m\times n}$ ($m
\geq n$) containing independent complex Gaussian entries with zero
mean and unit variance, and let $0<\lambda_1\leq \lambda_{2}\leq
\ldots\leq \lambda_n<\infty$ denote the eigenvalues of
$\mathbf{A}^{*}\mathbf{A}$ where $(\cdot)^*$ represents conjugate-transpose. This paper investigates the
distribution of the random variables $\frac{\sum_{j=1}^n
\lambda_j}{\lambda_k}$, for $k = 1$ and $k = 2$.  These two
variables are related to certain condition number metrics, including
the so-called Demmel condition number, which have been shown to
arise in a variety of applications. For both cases, we derive new
exact expressions for the probability densities, and establish the
asymptotic behavior as the matrix dimensions grow large.  In
particular, it is shown that as $n$ and $m$ tend to infinity with
their difference fixed, both densities scale on the order of $n^3$.
After suitable transformations, we establish exact expressions for
the asymptotic densities, obtaining simple closed-form expressions
in some cases. Our results generalize the work of Edelman on the
Demmel condition number for the case $m = n$.
\end{abstract}



\pagestyle{myheadings} \thispagestyle{plain} \markboth{P.
DHARMAWANSA, M. R. MCKAY and Y. CHEN}{ON DEMMEL AND RELATED CONDITION NUMBERS}

\section{Introduction}
Understanding the sensitivity of outcomes to initial conditions in certain
iterative algorithms is important for the design and
analysis of many physical systems. Example algorithms include iterative
methods used in linear algebra, interior-point methods of convex
optimization, and polynomial zero finding \cite{Cucker}. Various condition numbers have been defined as a measure of the sensitivity of the
solutions with respect to small perturbations of the input. One of the earliest studies by Turin \cite{Turing} on iterative algorithms related to matrix inversion and the solution of large systems of linear equations defined two such condition numbers as a measure of the degree of ill-conditioning in a matrix. In particular, for a matrix $\mathbf{A}\in \mathbb{R}^{n\times n}$, the first was referred to as the $N$-condition number, defined as $N(\mathbf{A})N(\mathbf{A}^{-1})/n$ with $N(\cdot)$ representing the Frobenius norm, whilst the second was referred to as the $M$-condition number, defined as $M(\mathbf{A})M(\mathbf{A}^{-1})/n$ with $M(\cdot)$ denoting the operator returning the largest absolute entry of the matrix. However, probably
the best known condition number, introduced in
\cite{Gold}, takes the form
$\kappa(\mathbf{A})=||\mathbf{A}||_2||\mathbf{A}^{-1}||_2$, where
$\mathbf{A}\in\mathbb{C}^{n\times n}$ and $||\cdot||_2$ denotes the
2-norm. This condition number, similar to the $N$-condition number and $M$-condition number, is important in problems involving
matrix inversion and solutions of linear equations.

As conjectured in \cite{Rene}, the calculation of a condition number pertaining to a problem with a certain deterministic data set is as hard as solving the problem itself with the data set. Whilst very difficult to prove in general, a proof of this conjecture has been established for {\emph{conic}} condition numbers in \cite{Cucker1}. To gain further insights into the behavior of the condition  number, probabilistic analysis based on certain probability measures on the data has been
employed; see
\cite{Cucker2,Che,Cucker3,Ed2,Ed1,Smale2,Smale1,Gold1} for a partial list. In
general, most existing literature dealing with the probabilistic
analysis of condition numbers has focused on one of two main
perspectives. The first is to establish bounds on the tail of the
distribution of $\kappa(\mathbf{A})$, which is important
for the geometrical characterization of the condition number \cite{Demmel}. Central to the geometric characterization is the fact that $\kappa(\mathbf{A})$ is proportional to the reciprocal of the distance to the set of ill-conditioned matrices \cite{Demmel}.
The second has been to establish bounds on the expectation of
$\ln(\kappa(\mathbf{A}))$, which is important for characterizing the
average loss of numerical precision and average running time
\cite{Cucker,Ed1} of iterative algorithms. In general, the ``randomness'' of the
condition number is introduced by selecting the elements of
$\mathbf{A}$ to be standard independent normal real/complex random
variables \cite{Ed2,Ed1,Smale1}.

Whilst in this paper we focus primarily on a probabilistic analysis,
it is worth noting that a deterministic version of the problems
considered above, related to the inversion of positive-definite
Hankel (or moment) matrices appearing naturally in random matrix
theory, involves the determination of the smallest eigenvalues of
Hankel matrices of order $n$. See \cite{berg-chen} and
\cite{berg,chen-law,chen-lu,sze,widom-wilf} for contributions
dealing with this problem.

Demmel, in his seminal paper \cite{Demmel}, introduced a new
condition number of \emph{conic} type, defined for square $n \times n$ matrix {\bf A} as
\begin{equation}
\label{dcn}
\kappa_D(\mathbf{A})=||\mathbf{A}||_F||\mathbf{A}^{-1}||_2\:,
\end{equation}
where $||\cdot||_F$ is the Frobenius norm. This, along with a
theorem due to Eckart and Young \cite{Young}, permits a geometrical
characterization \cite{Cucker,Demmel}, which has enabled the
calculation of bounds on the tail distribution of
$\kappa_D(\mathbf{A})$ with the help of integral geometry.  A random
variable similar to $\kappa_{D}(\mathbf{A})$ also arises in problems
dealing with the entanglement of a bi-partite quantum system
\cite{Ake, Chen}. As demonstrated in \cite{Dem2,Demmel},
the condition numbers arising in various contexts, including matrix
inversion, eigenvalue calculation, polynomial zero finding, as well
as pole assignment in linear control systems, can be bounded by $\kappa_D(\mathbf{A})$.
The definition (\ref{dcn}) extends naturally to rectangular matrices by replacing $\mathbf{A}^{-1}$ in (\ref{dcn}) with $\mathbf{A}^\dagger$ (also known as the Moore-Penrose inverse or the pseudo-inverse) to yield
\begin{equation}
\label{dcn1}
\kappa_D(\mathbf{A})=||\mathbf{A}||_F||\mathbf{A}^{\dagger}||_2.
\end{equation}
For $\mathbf{A}\in\mathbb{C}^{m\times n}$ with $\text{rank}(\mathbf{A})=r\;(\leq \min(m,n))$, (\ref{dcn1}) simplifies to
\begin{equation}
\label{dcngen}
\kappa_D(\mathbf{A})=\sqrt{\frac{\sum_{j=1}^r \lambda_j}{\lambda_1}}\;,
\end{equation}
where $\lambda_1\leq \lambda_{2}\leq
\ldots\leq \lambda_r$ are
the non-zero eigenvalues of $\mathbf{A}^*\mathbf{A}$ (or $\mathbf{A}\mathbf{A}^*)$ with $(\cdot)^*$ denoting conjugate-transpose.
In this paper, we deal exclusively with Gaussian random matrices $\mathbf{A}\in\mathbb{C}^{m\times n}$ ($m\geq n$) having independent complex standard normal entries\footnote{Henceforth, when referring to ``Gaussian matrices'', we will implicitly assume that the matrices have independent standard complex normal entries; i.e., this will not be explicitly stated.}. Therefore, (\ref{dcngen}) can be written as
\begin{equation}
\label{dcngen1}
\kappa_D(\mathbf{A})=\sqrt{\frac{\sum_{j=1}^n \lambda_j}{\lambda_1}}\;,
\end{equation}
which follows from the fact that the matrix $\mathbf{A}^*\mathbf{A}$ is positive-definite with probability one \cite{Eaton, Mallik}.
This form of the Demmel condition number was used by Krishnaiah
\emph{et.\ al.} in a sequence of papers \cite{Wal1,Wal2,Krish} to
derive the inferences on certain sub-hypotheses when the total
hypothesis, which is composed of various component sub-hypotheses, is rejected. Recently, it was shown in \cite{Cucker} that
the condition number of the form (\ref{dcn1}) is useful in
analyzing the loss of precision in the computation of the solution to the classical linear least squares problem.
Also, the statistical properties of the Demmel condition number
$\kappa_D(\mathbf{A})$ (for arbitrary $m$ and $n$) have recently
found important applications in the design and analysis of
contemporary wireless communication systems.  Specific examples
include the design of adaptive multi-antenna transmission techniques
\cite{Paul}, and the modeling of physical multi-antenna transmission channels \cite{Ost}.

From the discussion given above, there is clear motivation for
studying the statistical properties of the Demmel condition number
of random matrices.  Here we review some of the key existing
contributions dealing with this problem. These contributions include
\cite{Cucker,Demmel,Ed2,Ed3,Ed1}, which investigated the exact
distributions as well as bounds on the tail probabilities over
various regimes depending on the size of the random matrix.
For instance, Demmel showed that for Gaussian matrices
$\mathbf{A}\in\mathbb{C}^{n\times n}$ \cite{Demmel},
\[
\frac{(1-1/x)^{2n^2-2}}{2n^4x^2}<\Pr\left(\kappa_D(\mathbf{A})>x\right)<\frac{e^2n^5(1+n^2/x)^{2n^2-2}}{x^2}
\; ,
\]
while Edelman concluded that as $n\to\infty$ \cite{Ed3},
\[
\Pr\left(\frac{2\kappa_D(\mathbf{A})}{n^{\frac{3}{2}}}<x\right)\to
\exp\left(-\frac{4}{x^2}\right) \; .
\]
For Gaussian matrices $\mathbf{A}\in \mathbb{C}^{m\times n}$ with $m\geq n$, one could also exploit the bound $\kappa_D(\mathbf{A})\leq \sqrt{n} \kappa(\mathbf{A})$ (following from the fact that $||\mathbf{A}||_F\leq \sqrt{n} ||\mathbf{A}||_2$) along with the upper bound on the distribution tail of $\kappa(\mathbf{A})$ given in \cite[Theorem 4.6]{Dong} to derive an upper bound on the distribution tail of $\kappa_D(\mathbf{A})$. This, in turn, reveals that the $n^{3/2}$ asymptotic scaling order observed in \cite{Ed3} for square Gaussian matrices also serves as an upper bound on the scaling order for rectangular Gaussian matrices.  The exact scaling order, however, has yet to be determined.
In addition to these results, more recently, various exact closed-form expressions for the distribution of
the Demmel condition number for Gaussian matrices $\mathbf{A}\in \mathbb{C}^{m\times n}$ have been reported in \cite{Caj,Mat,Lu}. These results, however, are rather complicated and they become unwieldy
computationally when the matrix dimensions are not small. Moreover,
the expressions in \cite{Caj,Mat,Lu} are not suitable for
understanding the behavior of the scaled Demmel condition number as
the matrix dimensions grow large. In this respect, only for the case
$m = n$ given in \cite{Caj,Mat,Lu}, the asymptotic behavior of the
exact probability density function (p.d.f.) of the Demmel condition
number is known. Establishing the asymptotic properties for more
general $m\times n$ ($m\leq n$) matrices is one of the key
objectives to be addressed in this paper.

In addition to $\kappa_D(\mathbf{A})$, other related metrics of the
form $\lambda_k/\sum_{j=1}^n\lambda_j, k=1,2,\ldots,n$ have been considered by
Krishnaiah \emph{et.\ al.} \cite{Wal1,Wal2,Krish} in certain
hypothesis testing problems. Therein, they employed a result due to
Davis \cite{Davis}, which gives an expression for the p.d.f. of
$\lambda_k/\sum_{j=1}^n\lambda_j$ by establishing a Laplace
transform relationship with respect to the p.d.f.\ of $\lambda_k$.
However, the expressions obtained are very complicated, both
analytically and numerically. This complexity, in turn, does not easily facilitate the
understanding of the asymptotic behavior of the p.d.f.\ as the
matrix dimensions become large. In this paper, we tackle this
problem by adopting a different approach, making use of techniques
developed in \cite{Mehta}. We focus our analysis on the quantities, $\kappa_D^2(\mathbf{A})$ and
\[
\kappa_E^2(\mathbf{A})=\frac{\sum_{j=1}^n \lambda_j}{\lambda_2}
\; ,
\]
whilst noting that our approach may also pave the way for studying
more generalized metrics of the form
$\frac{\sum_{j=1}^n\lambda_j}{\lambda_k},k=1,2,\ldots,n$. We derive new
expressions for the exact and asymptotic distributions of
$\kappa_D^2(\mathbf{A})$ and $\kappa^2_E(\mathbf{A})$ for Gaussian matrices $\mathbf{A}\in \mathbb{C}^{m\times n}$ ($m\geq n$) by
adopting a moment generating function (m.g.f.) based approach. We
show the interesting result that both $\kappa_D^2(\mathbf{A})$ and
$\kappa_E^2(\mathbf{A})$ scale on the order of $n^3$ when $m$ and
$n$ tend to infinity in such a way that $m-n$ remains a fixed
integer. These results agree with and generalize the scaling
behavior obtained previously by Edelman in \cite{Ed3} for $n\times n$ Gaussian
matrices. The
scaled asymptotic p.d.f.\ which we derive for
$\kappa_D^2(\mathbf{A})$ is expressed in closed form for arbitrary
$m \geq n$, whilst for $\kappa_E^2(\mathbf{A})$ it involves a single
finite-range integral for the general case and a closed-form
solution for the scenario $m=n$.
\section{Preliminaries}
To facilitate our main derivations, we will require the following
preliminary results and definitions.
\begin{definition}
Let the elements of $\mathbf{A}\in\mathbb{C}^{m\times n}$ ($m\geq
n$) be independent and identically distributed complex standard
normal variables. Then the matrix $\mathbf{W}=\mathbf{A}^*
\mathbf{A}$ is said to follow a complex Wishart distribution, i.e.,
$\mathbf{W}\sim\mathcal{W}_n(m,\mathbf{I}_n)$.
\end{definition}

\begin{theorem}
The joint density of
the ordered eigenvalues $0< \lambda_1\leq \lambda_2\leq
\ldots\leq\lambda_n < \infty$ of $\mathbf{W}$ is given by \cite{JamesAT}
\begin{equation}
\label{wpdf}
f\left(\lambda_1,\lambda_2,\ldots,\lambda_n\right)=K_{n,\alpha}\;\Delta_n^2(\boldsymbol{\lambda})
\prod_{j=1}^n\lambda_j^{\alpha} e^{-\lambda_j}
\end{equation}
where, for
$\boldsymbol{\lambda}=\{\lambda_1,\lambda_2,\ldots,\lambda_n\}$,
$\Delta_n(\boldsymbol{\lambda}):=\prod_{1\leq j<k\leq
n}(\lambda_k-\lambda_j)$, $\alpha=m-n$, and ${K}_{n,\alpha}=n!
\left( \prod_{j=0}^{n-1}(j+1)!(j+\alpha)!\right)^{-1}$.
\end{theorem}
\begin{lemma}
For $\rho>-1,$ the generalized Laguerre polynomial of degree $N$,
$L_{N}^{(\rho)}(z)$, is defined by \cite{Erdelyi}:
\begin{eqnarray}
L_{N}^{(\rho)}(z) =\frac{(\rho+1)_N}{N!}{}_1F_1(-N,\rho+1,z) =\frac{(\rho+1)_N}{N!}\sum_{j=0}^N\frac{(-N)_j}{(\rho+1)_j}\frac{z^j}{j!}\;,\label{prop1}
\end{eqnarray}
with $k^{\rm th}$ derivative
\begin{eqnarray}
\frac{d^k}{dz^k}L_{N}^{(\rho)}(z)
=(-1)^kL_{N-k}^{(\rho+k)}(z)\;,\label{prop3}
\end{eqnarray}
where $(a)_j=a(a+1)\ldots(a+j-1)$ with  $(a)_0=1$ is
the Pochhammer symbol and ${}_1F_1(a;c;z)$ is the confluent hypergeometric function of the first kind.
\end{lemma}

\begin{lemma}\label{lem:stirling}
The monomial $z^n$ can be expanded in terms of the  Stirling number of the second kind,
$\mathtt{S}^{(m)}_n$, as follows \cite{Abromwitz}:
\begin{eqnarray}
z^n = \sum_{j=0}^n\mathtt{S}^{(j)}_n\;z(z-1)(z-2)\cdots (z-j+1), \;
\label{ster1}
\end{eqnarray}
where
\begin{equation}
\mathtt{S}^{(m)}_n =\frac{1}{m !} \sum_{j=0}^m (-1)^{m-j}\binom{m}{j}j^n,
\quad \mathtt{S}^{(0)}_0=\mathtt{S}^{(n)}_n =\mathtt{S}^{(1)}_n=1
\nonumber,
\end{equation}
and $\binom{m}{j}=\frac{m!}{j!(m-j)!}$.
\end{lemma}

Finally, we use the following notation to compactly represent the
determinant of an $N\times N$ block matrix:
\begin{equation}
\begin{split}
\det\left[a_{i,j}\;\; b_{i,k-2}\right]_{\substack{i=1,2,\ldots,N\\
 j=1,2\\
 k=3,4,\ldots,N}}&=\left|\begin{array}{cccccc}
 a_{1,1} & a_{1,2}& b_{1,1}& b_{1,2}& \ldots & b_{1,N-2}\\
  a_{2,1} & a_{2,2}& b_{2,1}& b_{2,2}& \ldots & b_{2,N-2}\\
  \vdots & \vdots & \vdots & \vdots &\ddots & \vdots \\
  a_{N,1} & a_{N,2}& b_{N,1}& b_{N,2}& \ldots & b_{N,N-2}
 \end{array}\right|.
 \end{split}
\end{equation}

\section{M.g.f. and p.d.f. of $\kappa_D^2(\mathbf{A})$}
In this section we derive new expressions for the m.g.f. and the
p.d.f. of $\kappa_D^2(\mathbf{A})$. Our approach follows along similar lines to \cite[Chap.\ 22]{Mehta}.

The m.g.f. of $\kappa_D^2(\mathbf{A})$ is given by
\begin{align}
\mathcal{M}_{\kappa_D^2(\mathbf{A})}(s) &= E \left[ e^{-s
\kappa_D^2(\mathbf{A})} \right] \nonumber \\
&= \int_{0}^{\infty}\int_{\mathcal{R}_1}
e^{-s\:\frac{\sum_{j=1}^{n}\lambda_j}{\lambda_1}}\:f(\lambda_1,\lambda_2,...,\lambda_n)d\lambda_2...d\lambda_n
d\lambda_1, \nonumber
\end{align}
where $\mathcal{R}_1=\{\lambda_1\leq \lambda_2\leq
\ldots\leq\lambda_n<\infty\}$.
Let $\lambda_1=x.$ An easy manipulation gives
\begin{equation}
\begin{split}
\label{eq:mgf}
\mathcal{M}_{\kappa_D^2(\mathbf{A})}(s) =e^{-s}\int_0^\infty \int_{\mathcal{R}_1}
& e^{-s\frac{\sum_{j=2}^n\lambda_j}{x}} f\left(x,\lambda_2,\ldots,\lambda_n\right)
d\lambda_2\ldots d\lambda_n\:dx.
\end{split}
\end{equation}
The theorem below gives a closed-form representation for the m.g.f.
\begin{theorem}\label{p1}
The m.g.f. of $\kappa_D^2(\mathbf{A})$ is given by
\begin{equation}
\begin{split}
\label{demgf}
\mathcal{M}_{\kappa_D^2(\mathbf{A})}(s) =\frac{n!\;e^{-ns}}{(m-1)!}
\int_0^\infty &
\frac{x^{mn-1}e^{-nx}}{(s+x)^{mn-\alpha-1}}\;\text{\em det}\left[L^{(l+1)}_{n+k-l-1}(-s-x)\right]_{k,l=1,2,\ldots,\alpha}dx.
\end{split}
\end{equation}
\end{theorem}
\begin{proof}
Substituting (\ref{wpdf}) into (\ref{eq:mgf}) we find, keeping the
$x$ integration last,
\begin{equation*}
\begin{split}
\mathcal{M}_{\kappa_D^2(\mathbf{A})}(s)=K_{n,\alpha}\;e^{-s}& \int_0^\infty
x^\alpha e^{-x}\left\{\int_{\mathcal{R}_1}
\prod_{j=2}^n(\lambda_j-x)^2\lambda_j^{\alpha} e^{-\lambda_j\left(1+\frac{s}{x}\right)}\right.\\
& \hspace{2.3cm}\left.\times
\prod_{2\leq i<j\leq n}(\lambda_j-\lambda_i)^2\;d\lambda_2\ldots d\lambda_n\right\}dx.
\end{split}
\end{equation*}
Relabeling variables as $\lambda_{j} = x_{j-1}, \, j=2,3,\ldots,n$,
and exploiting symmetry to remove the ordered region of integration
(i.e., achieved by dividing through by $(n-1)!$) gives
\begin{align}
\label{eq:inter}
\mathcal{M}_{\kappa_D^2(\mathbf{A})}(s)=\frac{K_{n,\alpha}\;e^{-s}}{(n-1)!}& \int_0^\infty
x^\alpha e^{-x}\left\{\int_{[x,\infty)^{n-1}}
\prod_{j=1}^{n-1}(x_j-x)^2x_j^{\alpha} e^{-x_j\left(1+\frac{s}{x}\right)}\right.\nonumber\\
& \hspace{3cm}\times
\Delta^2_{n-1}(\mathbf{x})\;dx_1dx_2\ldots dx_{n-1}\Biggr\}\;dx.
\end{align}
Now we apply the change of variables
$y_j=\frac{(x+s)}{x}\left(x_j-x\right), \, j=1,2,\ldots,n-1$ to the
inner $(n-1)$-fold integral in (\ref{eq:inter}) with some algebraic
manipulation to obtain
\begin{align}
\label{mgfbef}
\mathcal{M}_{\kappa_D^2(\mathbf{A})}(s)=\frac{K_{n,\alpha}\;e^{-ns}}{(n-1)!}
\int_0^\infty & \frac{x^{mn-1}e^{-nx}}{(x+s)^{mn-\alpha-1}}
(-1)^{\alpha(n-1)}Q(n-1,\alpha,-x-s)\:dx,
\end{align}
where we have defined
\begin{align}
\label{eq:Qdef} Q(n,\alpha,z) := \int_{[0,\infty)^n}
\Delta^2_{n}(\mathbf{y}) & \prod_{j=1}^{n} y_j^2
e^{-y_j}(z-y_j)^{\alpha}\; dy_1dy_2\ldots dy_{n} \; .
\end{align}

\newtheorem{rem}{Remark}

The remaining task is to obtain a closed-form solution to
$Q(n,\alpha,z)$. We point out that a solution to a generalization of
this integral, not restricting $\alpha$ to an integer, has been
derived in \cite{chenmckay,Ozipov2010} as a solution to a
Painlev\'{e} V equation. Here, for $\alpha$ an integer, we establish
a much simpler closed-form algebraic solution. For this purpose, we
employ a result from random matrix theory \cite[Section
22.2.2]{Mehta}, which gives
\begin{equation*}
\label{MT}
Q(n,\alpha,z)=\tilde b\;Q(n,0,z)\;\text{det}\left[\frac{d^l}{dz^l}C_{n+k}(z)\right]_{k,l=0,1,\ldots,\alpha-1}\:,
\end{equation*}
where $\tilde b=\left(\prod_{j=0}^{\alpha-1}j!\right)^{-1}.$ For our problem
$C_{j}(x)$ are monic polynomials orthogonal with respect to the
weight $x^2 e^{-x}$, over $0\leq x<\infty$. We see that $C_j(x)=(-1)^j\:j!\:L_j^{(2)}(x).$
Hence (\ref{eq:Qdef}) becomes
\begin{align}
\label{eq:int}
 Q(n,\alpha,z)& =\tilde b\;Q(n,0,z)\;\text{det}
\left[(-1)^{n+k}(n+k)!\frac{d^l}{dz^l}L^{(2)}_{n+k}(z)\right]_{k,l=0,1,\ldots,\alpha-1}\nonumber\\
&=\tilde b\;Q(n,0,z)(-1)^{n\alpha}\prod_{j=0}^{\alpha-1}(n+j)!\;
\text{det}
\left[L^{(2+l)}_{n+k-l}(z)\right]_{k,l=0,1,\ldots,\alpha-1}.
\end{align}
Moreover, we have $Q(n,0,z)=\prod_{j=0}^{n-1}(1+j)!(2+j)!$, which
when used with (\ref{eq:int}) in (\ref{mgfbef}), followed by
translating the indices from $k,l=0,1,\ldots,\alpha-1$ to
$k,l=1,2,\ldots,\alpha$ gives (\ref{demgf}).
\end{proof}

Now we take the inverse Laplace transform of
(\ref{demgf}) to arrive at the p.d.f.
of $\kappa_D^2(\mathbf{A})$, which is given below.
\begin{corollary}
The p.d.f. of $\kappa_D^2(\mathbf{A})$ is given by
\begin{align}
\label{pdfdemmel}
f^{(\alpha)}_{\kappa_D^2(\mathbf{A})}(y)=\frac{n!}{(n+\alpha-1)!}\frac{\Gamma(mn)}{y^{mn}}
\mathcal{L}^{-1}&\Biggl\{\frac{e^{-ns}}{s^{mn-\alpha-1}}\;
\text{\em det}\left[L^{(l+1)}_{n+k-l-1}(-s)\right]_{k,l=1,2,\ldots,\alpha}\Biggr\}
\end{align}
where $\mathcal{L}^{-1}(\cdot)$ denotes the inverse Laplace transform.
\end{corollary}

Interestingly, by noting that the p.d.f. of the minimum eigenvalue
$\lambda_1$ takes the form
\begin{align}
\label{mineigpdf}
f_{\min}(x)&=\frac{K_{n,\alpha}}{(n-1)!}\int_{[x,\infty)^{n-1}}f\left(x,\lambda_2,\ldots,\lambda_n\right)
d\lambda_2\ldots d\lambda_n\nonumber\\
&=
\frac{K_{n,\alpha}}{(n-1)!}x^{\alpha}e^{-nx}(-1)^{\alpha(n-1)}Q(n-1,\alpha,-x)\nonumber\\
& =
\frac{n!}{(n+\alpha-1)!}x^{\alpha} e^{-nx}\text{det}\left[L^{(l+1)}_{n+k-l-1}(-x)\right]_{k,l=1,2,\ldots,\alpha},
\end{align}
we can obtain the following alternative representation for the
p.d.f. of $\kappa_D^2(\mathbf{A})$:
\begin{equation*}
\label{demalt}
f^{(\alpha)}_{\kappa_D^2(\mathbf{A})}(y)=\frac{\Gamma(mn)}{y^{mn}}\mathcal{L}^{-1}\left\{\frac{f_{\min}(s)}{s^{mn-1}}\right\}.
\end{equation*}
This turns out to be a simpler representation of an equivalent
relation given previously in \cite{Krish} (obtaining one relation
from the other, however, appears to be non-trivial). This
simplified form results as a consequence of the m.g.f.\ derivation
approach, in contrast to the p.d.f.-based approach in \cite{Davis}.
We also mention that the minimum eigenvalue p.d.f.\
(\ref{mineigpdf}) fixes a sign problem with a result given in
\cite[Eq. 3.12]{Peter}.

The new expression (\ref{pdfdemmel}) facilitates the exact
evaluation of the p.d.f.\ of $\kappa_D^2(\mathbf{A})$ in
closed-form, for any value of $\alpha$. This is given in the
following key theorem:
\begin{theorem} \label{eq:ThExact}
The exact p.d.f. of $\kappa_D^2(\mathbf{A})$ is given by
\begin{align}
\label{pdfalpha}
f^{(\alpha)}_{\kappa_D^2(\mathbf{A})}(y)&=
\Gamma(mn)
\left(\prod_{k=0}^\alpha
\frac{n+k}{(k+1)!}\right)(y-n)^{mn-\alpha-2}y^{-mn}\nonumber\\
& \hspace{0.8cm}\times
\sum_{j_1=0}^{n+\alpha-2}
\ldots
\sum_{j_\alpha=0}^{n-1}\;
\left(\prod_{k=1}^\alpha
(-1)^{j_k}
\frac{(-n-\alpha+k+1)_{j_k}}{(k+2)_{j_k}\; j_k!}
(y-n)^{-j_k}\right)\\
&\hspace{4cm} \times
\frac{\Delta_\alpha(\mathbf{c})}
{\Gamma\left(mn-\alpha-1-\sum_{k=1}^\alpha j_k \right)} H(y-n)\:,\nonumber
\end{align}
where $\mathbf{c}=\{c_1(j_1),c_2(j_2),\ldots, c_\alpha(j_\alpha)\}$
with $c_l(j_l)=l+j_l$, and $H(z)$ denote the Heaviside unit step
function, i.e., $H(z)=1,\;z\geq 0,$ and $H(z)=0,\; z<0.$
\end{theorem}

\begin{proof}
We use (\ref{prop1}) to  write the determinant term in (\ref{pdfdemmel}) as
\begin{align}
\label{eq:rev}
&\text{det}\left[L^{(l+1)}_{n+k-l-1}(-s)\right]_{k,l=1,2,\ldots,\alpha}\nonumber\\
&= \prod_{k=1}^\alpha \frac{(n+k)!}{(k+1)!}
\text{det}\left[\frac{1}{(n+k-l-1)!}\sum_{j_l=0}^{n+k-l-1}\frac{(-n-k+l+1)_{j_l}}
{(l+2)_{j_l}
}\frac{(-s)^{j_l}}{j_l!}\right]_{k,l=1,2,\ldots,\alpha}.
\end{align}
Further manipulation in this form is difficult due to the dependence
of the summation upper limits on $k$ and $l$. To circumvent this
problem, we use the factorization
\begin{align}
\frac{(-n-k+l+1)_{j_l}} {(l+2)_{j_l} } &=
\frac{(-n-k+l+1)_{j_l}}{(-n-\alpha+l+1)_{j_l}}
\frac{(-n-\alpha+l+1)_{j_l}}{(l+2)_{j_l} } \nonumber \\
&= \frac{(n+k-l-1)!}{(n+\alpha-l-1)!}
\frac{(-n-\alpha+l+1)_{j_l}}{(l+2)_{j_l} }
\prod_{i=0}^{\alpha-k-1}(\tilde{c}_l-i)\nonumber
\end{align}
where $\tilde{c}_l=n+\alpha-1-j_l-l$, in (\ref{eq:rev}) with some algebraic manipulation to obtain
\begin{align}
\label{eqdet} &
\text{det}\left[L^{(l+1)}_{n+k-l-1}(-s)\right]_{k,l=1,2,\ldots,\alpha}
=\frac{(n+\alpha)!(n+\alpha-1)!}{n!
\;(n-1)!\prod_{k=1}^\alpha (k+1)!} \nonumber\\
& \times \sum_{j_1=0}^{n+\alpha-2} \ldots \sum_{j_\alpha=0}^{n-1}
\left(\prod_{k=1}^\alpha \frac{(-n-\alpha+k+1)_{j_k}}{(k+2)_{j_k}\;
j_k!} (-s)^{j_k}\right)
\text{det}\left[\prod_{i=0}^{\alpha-k-1}(\tilde{c}_l-i)\right]_{k,l=1,2,\ldots,\alpha}.
\end{align}
Now, invoking Lemma \ref{le:App} in the Appendix, then substituting
(\ref{eqdet}) into (\ref{pdfdemmel}) yields
\begin{align*}
f^{(\alpha)}_{\kappa_D^2(\mathbf{A})}(y)=
\Gamma(mn)
\left(\prod_{k=0}^\alpha
\frac{n+k}{(k+1)!}\right)y^{-mn}\;
&\sum_{j_1=0}^{n+\alpha-2}
\ldots
\sum_{j_\alpha=0}^{n-1}\;
\left(\prod_{k=1}^\alpha
(-1)^{j_k}
\frac{(-n-\alpha+k+1)_{j_k}}{(k+2)_{j_k}\; j_k!}
\right)\nonumber\\
& \times \Delta_\alpha(\mathbf{c})
\mathcal{L}^{-1}\left\{\frac{e^{-ns}}{s^{mn-1-\alpha-\sum_{k=1}^\alpha
j_k}}\right\}.
\end{align*}
Finally, the result (\ref{pdfalpha}) follows upon carrying out the
remaining Laplace inversion using \cite[Eq. 1.1.2.1]{Prud2}.
\end{proof}

For some small values of $\alpha$, (\ref{pdfalpha}) admits the following simple forms.
\begin{corollary}\label{cor:2}
The exact p.d.f.s of $\kappa_D^2(\mathbf{A})$ corresponding to
$\alpha=0$ and $\alpha=1$ are given, respectively, by
\begin{align*}
f_{\kappa_D^2(\mathbf{A})}^{(0)}(y)&=n(n^2-1)(y-n)^{n^2-2}y^{-n^2}H(y-n)\:,\\
f_{\kappa_D^2(\mathbf{A})}^{(1)}(y)&=
\frac{\Gamma(n(n+1))}{2!}n(n+1)(y-n)^{n(n+1)-3}y^{-n(n+1)}\\
&
\hspace{2.7cm}\times\sum_{i=0}^{n-1}(-1)^i
\frac{(-n+1)_i}{(3)_i\;i!}\frac{(y-n)^{-i}}{\Gamma(n(n+1)-i-2)}H(y-n).
\end{align*}
\end{corollary}
The expression for $\alpha=0$ agrees with a previous result given in
\cite{Ed3}.

\begin{rem}
Whilst previous equivalent expressions have been derived in
\cite{Caj,Mat,Lu}, the exact p.d.f. of $\kappa_D^2(\mathbf{A})$
given in (\ref{pdfalpha}) is a generalized and/or simpler
representation. Indeed, noting that the number of nested summations
depends only on $\alpha$, this formula provides an efficient way of
evaluating the p.d.f. of $\kappa_D^2(\mathbf{A})$, particularly for
small values of $\alpha$. Moreover, since the algebraic complexity
depends only on $n$ and the difference of $m$ and $n$, this in turn makes our result (\ref{pdfalpha}) very
useful for conducting an asymptotic analysis of
$\kappa_D^2(\mathbf{A})$ as $m$ and $n$ grow large, but their difference does not (something which
appears infeasible with previous expressions in \cite{Caj,Mat,Lu}).
This is the objective of the next section.
\end{rem}

\section{Asymptotic Characterization of $\kappa_D^2(\mathbf{A})$}
In this section, we employ the exact p.d.f.\ representation
(\ref{pdfalpha}) to investigate the distribution of
$\kappa_D^2(\mathbf{A})$, suitably scaled, for fixed $\alpha$ when
$m,n\to \infty$.  We have the following key result:
\begin{theorem}
\label{th:asylam1}
As $m$ and $n$ tend to $\infty$ such that $\alpha=m-n$ is fixed,
$\kappa_D^2(\mathbf{A})$ scales on the order of $n^3$. More
specifically, the scaled random variable
$V=\kappa_D^2(\mathbf{A})/\left({\mu n^3}\right)$, with $\mu\in\mathbb{R}^+$ an arbitrary constant, has the following asymptotic p.d.f.\
as $m$ and $n$ tend to $\infty$ with $\alpha=m-n$ fixed:
\begin{align} \label{eq:Th1Result}
f^{(\alpha)}_V(v) =\frac{e^{-\frac{1}{\mu v}}}{\mu v^2}&\;
\text{\em det}\left[
\sum_{i=0}^{k-1}
\sum_{j=0}^i
\mathtt{S}^{(j)}_i\binom{k-1}{i}l^{k-1-i}
\frac{I_{l+j+1}\left(\frac{2}{\sqrt{\mu v}}\right)}{(\mu v)^{\frac{j+1-l}{2}}}
\right]_{k,l=1,2,\ldots,\alpha}\hspace{-5mm}H(v)\:,
\end{align}
where
$$
I_n(z)=
\sum_{k=0}^{\infty}\;\frac{1}{k!(n+k)!}\;\left(\frac{z}{2}\right)^{n+2k}
$$
is the modified Bessel function of the first kind of order $n$.
\end{theorem}
\begin{proof}
We arrange the terms in (\ref{pdfalpha}), noting that $m=n+\alpha$, to obtain
\begin{align}
\label{pdfdemasy}
& f_{\kappa_D^2(\mathbf{A})}^{(\alpha)}(y)=
\Gamma(n(n+\alpha))
\left(\prod_{k=0}^\alpha
\frac{n+k}{(k+1)!}\right)\;
\left(1-\frac{n}{y}\right)^{n(n+\alpha)-\alpha-2}y^{-\alpha-2}\nonumber\\
& \qquad \quad  \times
\sum_{j_1=0}^{n+\alpha-2}
\ldots
\sum_{j_\alpha=0}^{n-1}\;
\left(\prod_{k=1}^\alpha
(-1)^{j_k}
\frac{(-n-\alpha+k+1)_{j_k}}{(k+2)_{j_k}\; j_k!}
\left(1-\frac{n}{y}\right)^{-j_k}y^{-j_k}\right)\nonumber\\
& \hspace{4.5cm} \times
\frac{\Delta_\alpha(\mathbf{c})}
{\Gamma\left(n(n+\alpha)-\alpha-1-\sum_{k=1}^\alpha j_k \right)} H(y-n).
\end{align}
Now we have to choose a suitable scaling for the variable
$\kappa_D^2(\mathbf{A})$ in terms of $n$ so that the above p.d.f.
converges as $n\to\infty$. Careful thought reveals that the term
$\left(1-\frac{n}{y}\right)^{n(n+\alpha)}$ converges to a finite
non-zero limit as $n\to\infty$ if we scale $\kappa_D^2(\mathbf{A})$
proportional to $n^3$. For this reason, we introduce the scaled
random variable $V=\kappa_D^2(\mathbf{A})/(\mu n^3)$, and focus on
the function $\mu n^3 f^{(3)}_{\mu n^3 V}(\mu n^3 v)$ as
$n\to\infty$. To this end, with $y=\mu n^3 v$, by using elementary
limiting arguments it can be shown that
\begin{align}
\label{sterdetexp} & \lim_{n\to\infty}\mu n^3 f^{(\alpha)}_{\mu n^3
V}(\mu n^3 v)
 \nonumber \\
 & \hspace*{1cm} =\left(\prod_{k=1}^\alpha\frac{1}{ (k+1)!}\right)\frac{e^{-\frac{1}{\mu v}}}{\mu^{\alpha+1}v^{\alpha+2}}
\sum_{j_1=0}^{\infty}
\ldots
\sum_{j_\alpha=0}^{\infty}\;
\left(\prod_{k=1}^\alpha
\frac{1}{(k+2)_{j_k}\; j_k!}
\frac{1}{(\mu v)^{j_k}}\right)
\Delta_\alpha(\mathbf{c}) H(v)\nonumber\\
& \hspace*{1cm}  = \left(\prod_{k=1}^\alpha\frac{1}{
(k+1)!}\right)\frac{e^{-\frac{1}{\mu
v}}}{\mu^{\alpha+1}v^{\alpha+2}} \text{det}\left[\sum_{j_l=0}^\infty
\frac{(l+j_l)^{k-1}}{(l+2)_{j_l}\; j_l!}\frac{1}{(\mu v)^{j_l}}
\right]_{k,l=1,2,\ldots,\alpha}H(v).
\end{align}
We now focus on simplifying the determinant. To this end, we use the
binomial theorem and the definition of the Stirling number
(\ref{ster1}) to arrive at
\begin{align*}
\text{det}&\left[\sum_{j_l=0}^\infty
\frac{(l+j_l)^{k-1}}{(l+2)_{j_l}\; j_l!}\frac{1}{(\mu v)^{j_l}}
\right]_{k,l=1,2,\ldots,\alpha}\nonumber\\
& = \text{det}\left[ \sum_{i=0}^{k-1}\sum_{j=0}^i
\binom{k-1}{i} l^{k-1-i} \mathtt{S}^{(j)}_i \right. \left.
\sum_{j_l=j}^\infty \frac{1}{(j_l-j)!(l+2)_{j_l}}\frac{1}{(\mu
v)^{j_l}} \right]_{k,l=1,2,\ldots,\alpha},
\end{align*}
which can be simplified upon noting the relations \cite{Watson}
\begin{align*}
\sum_{j_l=j}^\infty
\frac{1}{(j_l-j)!(l+2)_{j_l}}\frac{1}{(\mu v)^{j_l}}& =
\left(\frac{1}{\mu v}\right)^j
\frac{(l+1)!}{(l+j+1)!}\;
{}_0 F_1\left(-;l+j+2;\frac{1}{\mu v}\right),\\
{}_0 F_1\left(-;l+j+2;\frac{1}{\mu v}\right)& =(l+j+1)!(\mu v)^{\frac{l+j+1}{2}} I_{l+j+1}\left(\frac{2}{\sqrt{\mu v}}\right),
\end{align*}
with ${}_0 F_1\left(-;\rho;z\right)=\sum_{k=0}^\infty z^k/(\rho)_k k!$ denoting the generalized hypergeometric function,
to obtain
\begin{align}
&\text{det}\left[\sum_{j_l=0}^\infty
\frac{(l+j_l)^{k-1}}{(l+2)_{j_l}\; j_l!}\frac{1}{(\mu v)^{j_l}}
\right]_{k,l=1,2,\ldots,\alpha}\nonumber\\
& =\prod_{k=1}^\alpha (k+1)!\; \text{det}\left[
\sum_{i=0}^{k-1}\sum_{j=0}^i \binom{k-1}{i} l^{k-1-i}
\mathtt{S}^{(j)}_i\right. (\mu v)^{\frac{l-j+1}{2}}
I_{l+j+1}\left(\frac{2}{\sqrt{\mu v}}\right)
\Biggr]_{k,l=1,2,\ldots,\alpha}. \nonumber
\end{align}
Finally, using this result in (\ref{sterdetexp}) with some algebraic
manipulation, and noting that $\mu n^3 f^{(\alpha)}_{\mu n^3 V}(\mu
n^3 v)$ denotes the p.d.f. of the new variable $V$, concludes the
proof.
\end{proof}

\begin{rem}
Clearly, the above approach of deriving the p.d.f. of the asymptotic
scaled version of $\kappa_D^2(\mathbf{A})$ explicitly depends on the
availability of a closed-form expression for the p.d.f. of
$\kappa_D^2(\mathbf{A})$. Interestingly, one can directly manipulate
the m.g.f. of $\kappa_D^2(\mathbf{A})$ instead of the p.d.f. to
yield the same asymptotic density. Although we do not demonstrate it
here, we exploit this technique in deriving the scaled asymptotic
p.d.f. of $\kappa_E^2(\mathbf{A})$ in section 6.
\end{rem}

The exact asymptotic p.d.f. of the scaled
$\kappa_D^2(\mathbf{A})/(\mu n^3)$ takes the following simple forms
for small values of $\alpha$.

\begin{corollary}\label{prop2}

For $\alpha=0, 1, 2$, the result (\ref{eq:Th1Result}) becomes:
\begin{align*}
f^{(0)}_V(v)& =\frac{1}{\mu v^2} e^{-\frac{1}{\mu v}}H(v) , \quad
\quad \quad f^{(1)}_V(v)  =\frac{1}{\mu v^2} e^{-\frac{1}{\mu v}}I_{2}\left(\frac{2}{\sqrt{\mu v}}\right)H(v)\\
f^{(2)}_V(v)& =\frac{1}{\mu v^2} e^{-\frac{1}{\mu v}}\left\{I_{2}\left(\frac{2}{\sqrt{\mu v}}\right)
I_{4}\left(\frac{2}{\sqrt{\mu v}}\right)
-\left[I_{3}\left(\frac{2}{\sqrt{\mu v}}\right)\right]^2
+\sqrt{\mu v}I_{2}\left(\frac{2}{\sqrt{\mu v}}\right) I_{3}\left(\frac{2}{\sqrt{\mu v}}\right)\right\}H(v).
\end{align*}
\end{corollary}
The asymptotic p.d.f. corresponding to the case $\alpha=0$ and $\mu=1/4$ is given in \cite{Ed3}

The advantage of the asymptotic formula given in Theorem
\ref{th:asylam1} is that it provides an easy to use expression which
compares favorably with finite $n$ results. To further highlight this fact, in Fig. \ref{Fig1}, we compare the analytical asymptotic p.d.f. derived in Theorem \ref{th:asylam1} with simulated data points corresponding to $\alpha=1, n=50,\mu=4$ and $\alpha=2,n=50,\mu=4$.

\begin{figure}
 \centering
 \vspace*{1.0cm}
     \subfigure[$\alpha=1$]{
            \includegraphics[width=.7\textwidth]{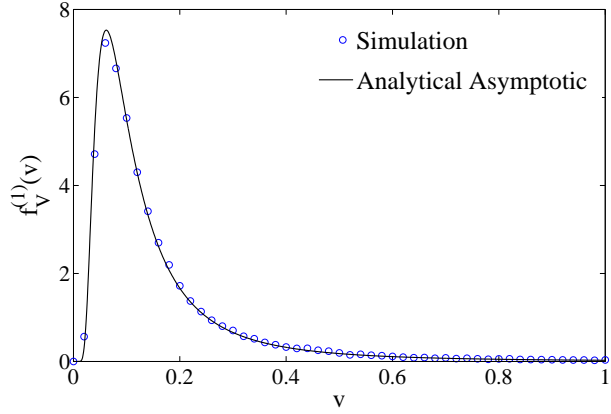}}
     \hspace{.3in}
     \subfigure[$\alpha=2$]{
          \includegraphics[width=.7\textwidth]{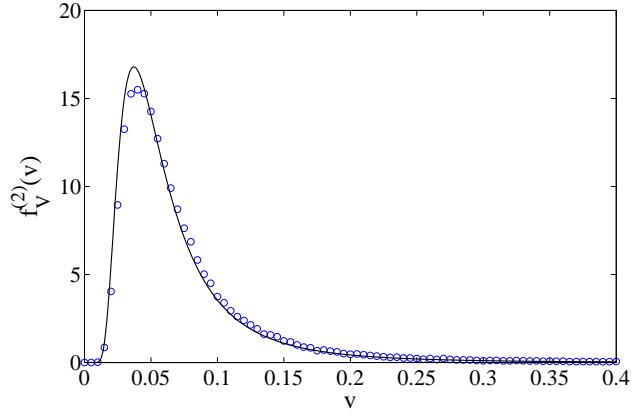}}\\
 \caption{Comparison of simulated data points and the analytical p.d.f.\ $f_V^{(\alpha)}(v)$ for $n=50$ with $\mu=4$.}
 \vspace*{0.3cm}
 \label{Fig1}
\end{figure}

Having characterized $\kappa_D^2(\mathbf{A})$, we now focus on
$\kappa_E^2(\mathbf{A})$.

\section{M.g.f. and p.d.f. of $\kappa_E^2(\mathbf{A})$}
Here we give new expressions for the m.g.f and the p.d.f. of
$\kappa_E^2(\mathbf{A})$. We point out that the key derivation steps
developed in this section may also be useful for characterizing the
distributional properties of more general condition number metrics,
given in \cite{Wal1,Wal2,Krish}.

Let $\lambda_2=x.$ By definition, the m.g.f. of $\kappa_E^2(\mathbf{A})$
for $n\geq 3$, is
\begin{align}
\label{lam2mgfdef}
\mathcal{M}_{\kappa_E^2(\mathbf{A})}(s)=e^{-s}\int_0^\infty
\Biggl\{\int_{\mathcal{R}_2}&\Biggl\{
\int_0^x
e^{-\frac{\lambda_1}{x}s-\frac{\sum_{j=3}^n\lambda_j}{x}s}
f\left(\lambda_1,x,\ldots,\lambda_n\right)\;d\lambda_1\Biggr\} d\lambda_3 d\lambda_4\ldots
d\lambda_n\Biggr\}dx,
\end{align}
where $\mathcal{R}_2=\{x\leq \lambda_3\leq \lambda_4 \leq \ldots
\leq \lambda_n\}$.  Analogous to the result given in (\ref{demgf}),
the following theorem provides an exact simple solution for this
m.g.f.:
\begin{theorem}
The m.g.f. of $\kappa_E^2(\mathbf{A})$, for $n\geq 3$, is given by
\begin{align}
\label{exactmgf}
\mathcal{M}_{\kappa_E^2(\mathbf{A})}(s) &=e^{-s(n-1)}\int_0^\infty
\frac{e^{-(n-1)x} x^{mn-1}}{(x+s)^{mn-4}}
\Biggl\{\int_0^1
\frac{z^2
 e^{-(1-z)\left(s+x\right)}}{(1-z)^{\alpha}}\\
 & \quad \times \det\left[L^{(j+1)}_{n+i-j-2}(-z(x+s))\;\; L^{(k-1)}_{n+i-k}(-(x+s))\right]_{\substack{i=1,2,\ldots,\alpha+2\\
 j=1,2\\
 k=3,4,\ldots,\alpha+2}}dz\Biggr\}dx\nonumber.
\end{align}
\end{theorem}
\begin{proof}
We use (\ref{wpdf}) in (\ref{lam2mgfdef}) with some manipulation to obtain
\begin{align*}
\mathcal{M}_{\kappa_E^2(\mathbf{A})}(s) & =K_{n,\alpha}\;e^{-s}\int_0^\infty
e^{-x} x^{\alpha}
\Biggl\{
\int_{\mathcal{R}_2}
e^{-\frac{\sum_{j=3}^n\lambda_j}{x}s}\nonumber\\
& \qquad \times
\prod_{j=3}^n
\lambda_j^\alpha e^{-\lambda_j}
\prod_{j=3}^n
(\lambda_j-x)^2
\prod_{3\leq k<j\leq n}
(\lambda_j-\lambda_k)^2\nonumber\\
& \qquad \times\left[
\int_0^x
(x-\lambda_1)^2
\prod_{j=3}^n (\lambda_j-\lambda_1)^2\lambda_1^\alpha e^{-\lambda_1\left(1+\frac{s}{x}\right)}
d\lambda_1\right]
d\lambda_3 \ldots d\lambda_n\Biggr\}dx\:,
\end{align*}
where we have used $
\Delta^2_n(\boldsymbol{\lambda})=(x-\lambda_1)^2
\prod_{j=3}^n(\lambda_j-\lambda_1)^2 \prod_{j=3}^n(\lambda_j-x)^2
\prod_{3\leq k<j\leq n} (\lambda_j-\lambda_k)^2$, which is valid for $n\geq
3$. Applying the variable transformation $\lambda_1=x z$ to the
innermost integral yields
\begin{align*}
\mathcal{M}_{\kappa_E^2(\mathbf{A})}(s) & =K_{n,\alpha}e^{-s}\int_0^\infty
e^{-x} x^{2m-1}
\Biggl\{
\int_{\mathcal{R}_2}
e^{-\frac{\sum_{j=3}^n\lambda_j}{x}s}
\prod_{j=3}^n
\lambda_j^\alpha e^{-\lambda_j}
\prod_{j=3}^n
(\lambda_j-x)^2
\nonumber\\
& \hspace{2cm}\times \prod_{3\leq k<j\leq n} (\lambda_j-\lambda_k)^2
\varphi\left(\frac{\lambda_3}{x},\frac{\lambda_4}{x},\ldots,\frac{\lambda_n}{x},x\right)
d\lambda_3 \ldots d\lambda_n\Biggr\}\;dx,
\end{align*}
where we have defined:
\begin{align*}
\varphi\left(\frac{\lambda_3}{x},\frac{\lambda_4}{x},\ldots,\frac{\lambda_n}{x},x\right):=
\int_0^1 z^\alpha (1-z)^2 \prod_{j=3}^n
\left(\frac{\lambda_j}{x}-z\right)^2 e^{-z\left(s+x\right)} dz \; \;
.
\end{align*}
By symmetry, we convert the ordered region of integration to an
unordered region, and subsequently introduce the variable
transformations $y_j=\frac{(x+s)}{x}\left(\lambda_j-x\right),\;
j=3,4,\ldots,n$, to obtain
\begin{align*}
\mathcal{M}_{\kappa_E^2(\mathbf{A})}(s) & =\frac{K_{n,\alpha}}{(n-2)!}e^{-s(n-1)}\int_0^\infty
\frac{e^{-(n-1)x}x^{mn-1}}{(x+s)^{mn-2m}}\nonumber\\
& \qquad \times
\Biggl\{
\int_{[0,\infty)^{n-2}}
\prod_{j=3}^n
y_j^2(y_j+x+s)^{\alpha} e^{-y_j}
\prod_{3\leq k<j\leq n}
(y_j-y_k)^2\nonumber\\
& \qquad \times \left[\int_0^1 z^\alpha (1-z)^2 \prod_{j=3}^n
\left(\frac{y_j}{x+s}+1-z\right)^2 e^{-z\left(s+x\right)} dz\right]
dy_3 \ldots dy_n\Biggr\}dx.
\end{align*}
Now, changing the order of the innermost integrals and then
relabeling the variables according to $y_j=x_{j-2},\;
j=3,4,\ldots,n$, yields
\begin{align}
\label{mgf} \mathcal{M}_{\kappa_E^2(\mathbf{A})}(s) &
=\frac{K_{n,\alpha}}{(n-2)!}e^{-s(n-1)}\int_0^\infty
\frac{e^{-(n-1)x} x^{mn-1}}{(x+s)^{mn-2\alpha-4}} \Biggl[\int_0^1
z^\alpha (1-z)^2
 e^{-z\left(s+x\right)}\nonumber\\
 & \qquad \qquad \times (-1)^{n\alpha}
R(n-2,-(1-z)(x+s),-(x+s),\alpha)dz\Biggr]dx,
\end{align}
where we have defined
\begin{align*}
R(n,a,b,\alpha) := \int_{[0,\infty)^{n}} \prod_{j=1}^{n} x_j^2
e^{-x_j}&\left(a-x_j\right)^2 (b-x_j)^{\alpha}
\Delta^2_{n}(\mathbf{x})\;dx_1 \ldots dx_{n} \; .
\end{align*}

The remainder of the proof is focused on evaluating
$R(n,a,b,\alpha)$. Following \cite[Eqs. 22.4.2, 22.4.11]{Mehta}, we
start with the related integral
\begin{align}
\label{eqmehta1}
\int_{[0,\infty)^n}
\prod_{j=1}^n
x_j^2 e^{-x_j}& \prod_{i=1}^{\alpha+2} (r_i-x_j) \Delta^2_{n}(\mathbf{x})\;dx_1dx_2\ldots dx_n\nonumber\\
& \qquad \qquad =\frac{n!}{K_{n,2}}\Delta^{-1}_{\alpha+2}(\mathbf{r})
\det\left[C_{n+i-1}(r_j)\right]_{i,j=1,2,\ldots,\alpha+2},
\end{align}
where $C_k(x)$ are monic polynomials orthogonal with respect to
$x^2e^{-x}$, over $0\leq x<\infty$. As such,
$C_k(x)=(-1)^kk!L^{(2)}_k(x)$, which upon substituting into
(\ref{eqmehta1}) gives
\begin{align}
\label{eqmehta}
& \int_{[0,\infty)^n}
\prod_{j=1}^n
x_j^2 e^{-x_j} \prod_{i=1}^{\alpha+2} (r_i-x_j) \Delta^2_{n}(\mathbf{x})\;dx_1dx_2\ldots dx_n\nonumber\\
&\qquad=
\frac{n!}{K_{n,2}}(-1)^{(n-1)\alpha}\prod_{j=0}^{\alpha+1}(-1)^{j+1}(n+j)!
\frac{\det\left[L^{(2)}_{n+i-1}(r_j)\right]_{i,j=1,2,\ldots,\alpha+2}}{\Delta_{\alpha+2}(\mathbf{r})}
\; .
\end{align}
In the above formula, the $r_i$s are, in general, distinct
parameters. However, if we can choose $r_i$ such that
\begin{align*}
r_i=\Biggl\{ \begin{array}{ll}
a & \text{if $i=1,2$}\\
b & \text{if $i=3,4,\ldots,\alpha+2$},
\end{array}
\end{align*}
then the left side of (\ref{eqmehta}) becomes precisely
$R(n,a,b,\alpha)$. Under direct substitution, however, it turns out
that the right-hand side of (\ref{eqmehta}) then gives a
$\frac{0}{0}$ indeterminate form. Therefore, the task is to evaluate
the following limits:
\begin{align}
\label{reval} R(n,a,b,\alpha)&
=\frac{n!}{K_{n,2}}(-1)^{(n-1)\alpha}\prod_{j=0}^{\alpha+1}(-1)^{j+1}(n+j)! \nonumber \\
& \times \lim_{\substack{r_1,r_2\to
a\\r_3,r_4,\ldots,r_{\alpha+2}\to b}}
\frac{\det\left[L^{(2)}_{n+i-1}(r_j)\right]_{i,j=1,2,\ldots,\alpha+2}}{\Delta_{\alpha+2}(\mathbf{r})}.
\end{align}
This can be solved by capitalizing on an approach outlined in
\cite{PeterJ} which gives
\begin{align}
\label{detlim}
 \lim_{\substack{r_1,r_2\to a\\r_3,r_4,\ldots,r_{\alpha+2}\to b}}&
 \frac{\det\left[L^{(2)}_{n+i-1}(r_j)\right]_{i,j=1,2,\ldots,\alpha+2}}{\Delta_{\alpha+2}(\mathbf{r})}\nonumber\\
 & \qquad =
 \frac{\det\left[\displaystyle\frac{d^{j-1}}{d a^{j-1}}L^{(2)}_{n+i-1}(a)\;\;\;\frac{d^{k-3}}{d b^{k-3}}L^{(2)}_{n+i-1}(b) \right]_{\substack{i=1,2,\ldots,\alpha+2\\
 j=1,2\\
 k=3,4,\ldots,\alpha+2}}}
 {\det\left[\displaystyle \frac{d^{j-1}}{da^{j-1}}a^{i-1}\;\;\; \frac{d^{k-3}}{db^{k-3}}b^{i-1}\right]
 _{\substack{i=1,2,\ldots,\alpha+2\\
 j=1,2\\
 k=3,4,\ldots,\alpha+2}}}.
\end{align}
The determinant in the denominator of (\ref{detlim}) can be written as
\begin{align*}
\det\left[\displaystyle \frac{d^{j-1}}{da^{j-1}}a^{i-1}\;\;\; \frac{d^{k-3}}{db^{k-3}}b^{i-1}\right]
 _{\substack{i=1,2,\ldots,\alpha+2\\
 j=1,2\\
 k=3,4,\ldots,\alpha+2}}=\prod_{j=0}^{\alpha-1}j!(a-b)^{2\alpha}.
\end{align*}
The numerator can also be simplified using (\ref{prop3}) to yield
\begin{align*}
&\det\left[\displaystyle\frac{d^{j-1}}{d a^{j-1}}L^{(2)}_{n+i-1}(a)\;\;\;\frac{d^{k-3}}{d b^{k-3}}L^{(2)}_{n+i-1}(b) \right]_{\substack{i=1,2,\ldots,\alpha+2\\
 j=1,2\\
 k=3,4,\ldots,\alpha+2}}\nonumber\\
 & \hspace*{1cm} =
 (-1)^{-\alpha}\prod_{j=0}^{\alpha+1}(-1)^{j+1}
 \det\left[L^{(j+1)}_{n+i-j}(a)\;\; L^{(k-1)}_{n+i-k+2}(b)\right]_{\substack{i=1,2,\ldots,\alpha+2\\
 j=1,2\\
 k=3,4,\ldots,\alpha+2}}.
\end{align*}
Substituting these expressions into (\ref{detlim}) and then the
result into (\ref{reval}) gives
\begin{align}
\label{Qans} R(n,a,b,\alpha)&
=\frac{(-1)^{-n\alpha}}{(a-b)^{2\alpha}}\frac{n!}{K_{n,2}}\frac{\prod_{j=0}^{\alpha+1}(n+j)!}{\prod_{j=0}^{\alpha-1}j!}
\det\left[L^{(j+1)}_{n+i-j}(a)\;\; L^{(k-1)}_{n+i-k+2}(b)\right]_{\substack{i=1,2,\ldots,\alpha+2\\
 j=1,2\\
 k=3,4,\ldots,\alpha+2}}.
\end{align}
Now we use this result in (\ref{mgf}), apply the relation $\int_0^a
f(x) dx=\int_0^a f(a-x) dx$ in the inner integral (i.e., with
respect to $z$), along with some basic manipulations which concludes
the proof.
\end{proof}

A straightforward Laplace inversion of the m.g.f. of
$\kappa_E^2(\mathbf{A})$ now gives the p.d.f. of
$\kappa_E^2(\mathbf{A})$, which is presented in the following
corollary:
\begin{corollary}
The p.d.f. of $\kappa_E^2(\mathbf{A})$, for $n\geq 3$, is given by
\begin{align}
\label{pdf}
f^{(\alpha)}_{\kappa_E^2(\mathbf{A})}(y)& =\frac{\Gamma(mn)}{y^{mn}}
\mathcal{L}^{-1}\Biggl\{\frac{e^{-s(n-1)}}{s^{mn-4}}\int_0^1
z^2(1-z)^{-\alpha}
e^{-(1-z)s}\nonumber\\
&\qquad \qquad
\times
\det\left[L^{(j+1)}_{n+i-j-2}(-sz)\;\; L^{(k-1)}_{n+i-k}(-s)\right]_{\substack{i=1,2,\ldots,\alpha+2\\
 j=1,2\\
 k=3,4,\ldots,\alpha+2}}dz\Biggr\}.
\end{align}
\end{corollary}

Although further simplification of (\ref{pdf}) seems intractable for
general matrix dimensions $m$ and $n$, we can obtain a closed-form
solution in the important case of square Gaussian matrices (i.e.,
$m=n$). This is given as follows:
\begin{corollary}
For $\alpha = 0$, (\ref{pdf}) becomes
\begin{align}
 f^{(0)}_{\kappa_E^2(\mathbf{A})}(y) =\frac{\Gamma(n^2)}{12}& n^2(n^2-1)
\sum_{i=0}^{n-1}\sum_{j=0}^{n-2}
\frac{(-n+1)_i(-n+2)_j}{(3)_i(4)_j\;i!j!}(j+1-i)(i+j+2)!\nonumber\\
&  \times \left(\sum_{k=0}^{i+j+2}\frac{(-1)^{i+j+k}}{(i+j+2-k)!}
\frac{y^{-n^2}(y-n+1)^{n^2+k-i-j-4}}{\Gamma\left(n^2+k-i-j-3\right)}
H(y-n+1)\right.\nonumber\\
&\hspace{3.3cm}-\left.\frac{y^{-n^2}(y-n)^{n^2-2}}{\Gamma(n^2-1)}H(y-n) \right).
\end{align}
\end{corollary}

It turns out that there is also an interesting connection
between the density of $\kappa_E^2(\mathbf{A})$ and the density of
the second smallest eigenvalue, $\lambda_2$.  To see this, we start
by writing the p.d.f.\ of $\lambda_2$ as
\begin{align*}
f_{\lambda_2}(x)&=\int_{\mathcal{R}_2} \int_0^x
f\left(\lambda_1,x,\ldots,\lambda_n\right)\;d\lambda_1 d\lambda_3
d\lambda_4\ldots d\lambda_n ,
\end{align*}
by definition.  This can be further simplified by using very similar
techniques to those used in the above m.g.f.\ derivation (we omit
the specific details) to yield
\begin{align*}
f_{\lambda_2}(x)=\frac{K_{n,\alpha}}{(n-2)!}x^{2\alpha+3}e^{-(n-1)x}
\int_0^1 &z^{\alpha}(1-z)^{2} e^{-zx}(-1)^{n\alpha}R(n-2,-(1-z)x,-x,\alpha) dz.
\end{align*}
Now, applying (\ref{Qans}), we obtain the following expression for
the p.d.f.\ of $\lambda_2$:
\begin{align}
\label{pdflam2}
f_{\lambda_2}(x) =x^3 e^{-(n-1)x}\int_0^1
& z^{2}
(1-z)^{-\alpha} e^{-(1-z)x}\nonumber\\
&
\times
\det\left[L^{(j+1)}_{n+i-j-2}(-xz)\;\; L^{(k-1)}_{n+i-k}(-x)\right]_{\substack{i=1,2,\ldots,\alpha+2\\
 j=1,2\\
 k=3,4,\ldots,\alpha+2}}dz.
\end{align}
An equivalent expression can also be obtained by starting with the
joint density of $\lambda_1$ and $\lambda_2$ given in \cite{Peter}.

\begin{rem}
We remark that the integrands corresponding to variable $z$ in (\ref{exactmgf}), (\ref{pdf}) and (\ref{pdflam2}) are well defined in the vicinity of $z=1$.
\end{rem}

Combining (\ref{pdf}) and (\ref{pdflam2}), we obtain the following
interesting connection between the densities of
$\kappa_E^2(\mathbf{A})$ and $\lambda_2$:
\begin{align} \label{eq:Connect}
f^{(\alpha)}_{\kappa_E^2(\mathbf{A})}(y)=\frac{\Gamma(mn)}{y^{mn}}\mathcal{L}^{-1}
\left\{\frac{f_{\lambda_2}(s)}{s^{mn-1}}\right\}.
\end{align}
A previous connection between these two densities was also
established in \cite{Krish}, but the result obtained was more
complicated. Moreover, it appears difficult to establish our
simplified connection (\ref{eq:Connect}) as a consequence of the
representation in \cite{Krish}. Our simplified representation here
was made possible by adopting an alternative derivation approach
based on the m.g.f., in contrast to the p.d.f.-based approach used
in \cite{Krish}.

It is also worth contrasting the expression (\ref{pdflam2}) with
prior equivalent results in the literature, particularly those
presented in \cite{Luis,Zan}. First, the result (\ref{pdflam2}) is
more compact, and in contrast to prior results in \cite{Luis,Zan} it
does not involve summations of functions over combinatorial partitions of
integers. Moreover, whilst both (\ref{pdflam2}) and the previous
expressions in \cite{Luis,Zan} involve determinants, the
representations are markedly different---in (\ref{pdflam2}), the
size of the determinant depends on the \emph{difference} between $m$
and $n$ (i.e., $\alpha$), whilst in \cite{Luis,Zan} it depends on
$m$ explicitly. (Determinantal expressions having a similar
$m$-dependence have also been derived for various more complicated
random matrix ensembles; for example, complex non-central Wishart
matrices \cite{JinMcKTCom} and generalized-$F$ matrices
\cite{SunMckTSP}.) This, in turn, gives remarkably simplified
expressions in our case when $m$ and $n$ are of roughly the same
order.  For example, when $\alpha = 0$ (i.e., $m = n$), we have the
following simple closed-form expression:
\begin{corollary}
For $\alpha=0$, (\ref{pdflam2}) evaluates to
\begin{align}
 f_{\lambda_2}(x)  =\frac{1}{12}n^2(n^2-1)e^{-(n-1)x}&
\sum_{i=0}^{n-1}\sum_{j=0}^{n-2}
\frac{(-n+1)_i(-n+2)_j}{(3)_i(4)_j\;i!j!}(j+1-i)(i+j+2)!\nonumber\\
& \times \left(\sum_{k=0}^{i+j+2}\frac{(-1)^{i+j+k}}{(i+j+2-k)!}x^{i+j+2-k}-e^{-x}\right)
.
\end{align}
\end{corollary}

Having analyzed $\kappa_E^2(\mathbf{A})$ for finite values of $m$
and $n$, we now focus our attention on the asymptotic behavior of
the scaled version of $\kappa_E^2(\mathbf{A})$.

\section{Asymptotic characterization of $\kappa_E^2(\mathbf{A})$}

Here we investigate how $\kappa_E^2(\mathbf{A})$ scales as $m$ and
$n$ tend to $\infty$ with $m-n=\alpha$ being fixed. Unlike in the
previous asymptotic analysis, in this case it is most convenient to
manipulate the m.g.f.\ of $\kappa_E^2(\mathbf{A})$, rather than the
p.d.f.\ We have the following key result:
\begin{theorem}\label{thm:keasy}
As $m$ and $n$ tend to $\infty$ such that $\alpha=m-n$ is fixed,
$\kappa_E^2(\mathbf{A})$ scales on the order of $n^3$. More
specifically, the scaled random variable
$V=\kappa_E^2(\mathbf{A})/(\mu n^3)$, with $\mu\in\mathbb{R}^+$ being an
arbitrary constant, has the following asymptotic p.d.f.
as $m$ and $n$ tend to $\infty$ with $\alpha=m-n$ fixed:
\begin{align}
\label{scaledpdf}
 f^{(\alpha)}_V(v)=\frac{e^{-\frac{1}{\mu v}}}{\mu ^4 v^5}
\int_0^1 &
\frac{z^2}{(1-z)^{\alpha}}\; \det\left[g_{i,j}(z,v)\;\;\; g_{i,k-2}(1,v)\right]_{\substack{i=1,2,\ldots,\alpha+2\\
 j=1,2\\
 k=3,4,\ldots,\alpha+2}}\:dz,
\end{align}
where, for $i,j=1,2,\ldots,\alpha+2$,
\begin{align*}
g_{i,j}(z,v)=\sum_{p=0}^{i-1}\sum_{q=0}^p\mathtt{S}^{(q)}_p\binom{i-1}{p}&j^{i-1-p}
\left(\frac{z}{\mu
v}\right)^{\frac{q-j-1}{2}}I_{q+j+1}\left(2\sqrt{\frac{z}{\mu v}}
\right) \; .
\end{align*}
\end{theorem}
\begin{proof}
Our strategy is to derive the m.g.f. of $V$ using the m.g.f.\ of
$\kappa_E^2(\mathbf{A})$ given in (\ref{exactmgf}). Subsequent
application of the limits on $m$ and $n$ followed by the Laplace
inversion will then yield the desired asymptotic p.d.f. As such, we
can write
\begin{equation}
\label{MGFasydef} \mathcal{M}_{\frac{\kappa_E^2(\mathbf{A})}{\mu
n^3}}(s)=\mathcal{M}_{\kappa_E^2(\mathbf{A})}\left(\frac{s}{\mu
n^3}\right).
\end{equation}
We first employ (\ref{prop1}), and use similar arguments as in the
derivation of (\ref{eqdet}), to write the determinant in (\ref{exactmgf}) as
\begin{align}
 \det & \left[L^{(j+1)}_{n+i-j-2}(-z(x+s))\;\; L^{(k-1)}_{n+i-k}(-(x+s))\right]_{\substack{i=1,2,\ldots,\alpha+2\\
 j=1,2\\
 k=3,4,\ldots,\alpha+2}}\nonumber\\
 & =\frac{\prod_{j=1}^{\alpha+2}(n+j-1)!\left(\prod_{j=3}^{\alpha+2}(j-1)!\right)^{-1}}{2!3!(n+\alpha-1)!(n+\alpha-2)!\prod_{j=3}^{\alpha+2}(n+\alpha+2-j)!}\nonumber\\
 & \quad \times
 \sum_{l_1=0}^{n+\alpha-1}\sum_{l_2=0}^{n+\alpha-2}\sum_{l_3=0}^{n+\alpha-1}\sum_{l_4=0}^{n+\alpha-2}\sum_{l_5=0}^{n+\alpha-3}\ldots
 \sum_{l_{\alpha+2}=0}^{n}\frac{(-n-\alpha+1)_{l_1}(-n-\alpha+2)_{l_2}}{(3)_{l_1}(4)_{l_2}l_1!l_2!}\nonumber\\
 & \hspace{1cm}\qquad \times
 z^{l_1+l_2}(-(x+s))^{l_1+l_2}\left(\prod_{k=1}^{\alpha}
 \frac{(-n-\alpha+k)_{l_{k+2}}}{(k+2)_{l_{k+2}}l_{k+2}!}(-(x+s))^{l_{k+2}}\right)\nonumber\\
  & \qquad \hspace{1cm} \times \det\left[\prod_{l=0}^{\alpha+1-i}(\tilde{z}_j-l)\;\;\;\prod_{l=0}^{\alpha+1-i}(\tilde{w}_k-l)\right]_{\substack{i=1,2,\ldots,\alpha+2\\
 j=1,2\\
 k=3,4,\ldots,\alpha+2}}
\end{align}
where $\tilde{z}_j=n+\alpha-j-l_j$ and $\tilde{w}_k=n+\alpha+2-k-l_k$.
Now, invoking Lemma \ref{lem:det2} to compute the remaining
determinant, we use the resulting expression along with
(\ref{exactmgf}) in (\ref{MGFasydef}) to obtain the m.g.f. of
$\frac{\kappa_E^2(\mathbf{A})}{\mu n^3}$ as
\begin{align*}
 \mathcal{M}_{\frac{\kappa_E^2(\mathbf{A})}{\mu n^3}}(s) &=
\frac{(n+\alpha)^2(n+\alpha+1)(n+\alpha-1)}{2!3!\prod_{j=3}^{\alpha+2}(j-1)!}e^{-\frac{s}{\mu n^3}(n-1)}\nonumber\\
& \times \int_0^\infty \frac{e^{-(n-1)x}
x^{n(n+\alpha)-1}}{\left(x+\frac{s}{\mu n^3}\right)^{n(n+\alpha)-4}}
\Biggl\{\int_0^1 (1-z)^{-\alpha} z^2
 e^{-(1-z)\left(\frac{s}{\mu n^3}+x\right)}\nonumber\\
 & \times \sum_{l_1=0}^{n+\alpha-1}\sum_{l_2=0}^{n+\alpha-2}\sum_{l_3=0}^{n+\alpha-1}\sum_{l_4=0}^{n+\alpha-2}
 \sum_{l_5=0}^{n+\alpha-3}\ldots
 \sum_{l_{\alpha+2}=0}^{n}\frac{(-n-\alpha+1)_{l_1}(-n-\alpha+2)_{l_2}}{(3)_{l_1}(4)_{l_2}l_1!l_2!}z^{l_1+l_2}\nonumber\\
  & \times
 \left(\prod_{k=1}^{\alpha}
 \frac{(-n-\alpha+k)_{l_{k+2}}}{(k+2)_{l_{k+2}}l_{k+2}!}\left(-\left(x+\frac{s}{\mu n^3}\right)\right)^{l_{k+2}}\right) \left(-\left(x+\frac{s}{\mu n^3}\right)\right)^{l_1+l_2}\Delta_{\alpha+2}(\mathbf{w})\;dz\Biggr\}
 \;dx.
\end{align*}
Then, we use the variable transformation $t=xn$, change the order of
integration, and apply elementary limiting arguments as $n \to
\infty$ to obtain
\begin{align*}
\lim_{n\to\infty}\mathcal{M}_{\frac{\kappa_E^2(\mathbf{A})}{\mu n^3}}(s)&=
\frac{1}{2!3!\prod_{j=3}^{\alpha+2}(j-1)!}
\int_0^1
(1-z)^{-\alpha}
z^2
\sum_{l_1=0}^{\infty}\sum_{l_2=0}^{\infty}\ldots
 \sum_{l_{\alpha+2}=0}^{\infty}\frac{z^{l_1+l_2}\Delta_{\alpha+2}(\mathbf{w})}{(3)_{l_1}(4)_{l_2}l_1!l_2!}\nonumber\\
 & \qquad \times
 \left(\prod_{k=1}^{\alpha}
 \frac{1}{(k+2)_{l_{k+2}}l_{k+2}!}\right) \left\{
 \int_0^\infty
 e^{-t-\frac{s}{\mu t}}t^{3+\sum_{j=1}^{\alpha+2}l_j}\;dt\right\}\;dz.
\end{align*}
Using the variable transformation $x=\frac{1}{\mu t}$ we compute the
inner integral, and subsequently
perform a Laplace inversion to yield the limiting p.d.f.\ of
$V=\frac{\kappa_E^2(\mathbf{A})}{\mu n^3}$:
\begin{align*}
f^{(\alpha)}_V(v) &=
\frac{1}{2!3!\prod_{j=3}^{\alpha+2}(j-1)!\;}\frac{e^{-\frac{1}{\mu
v}}}{\mu^4 v^5} \int_0^1 (1-z)^{-\alpha} z^2
\sum_{l_1=0}^{\infty}\sum_{l_2=0}^{\infty}\ldots
 \sum_{l_{\alpha+2}=0}^{\infty}\frac{z^{l_1+l_2}\Delta_{\alpha+2}(\mathbf{w})}{(3)_{l_1}(4)_{l_2}l_1!l_2! \;(\mu v)^{l_1+l_2}}\nonumber\\
 & \qquad \qquad\times
 \left(\prod_{k=1}^{\alpha}
 \frac{1}{(k+2)_{l_{k+2}}l_{k+2}!\;(\mu v)^{l_{k+2}}}\right)dz
 \nonumber \\
&=
\frac{1}{2!3!\prod_{j=3}^{\alpha+2}(j-1)!\;}\frac{e^{-\frac{1}{\mu
v}}}{\mu^4 v^5} \int_0^1
(1-z)^{-\alpha}z^2\nonumber\\
& \qquad \times \det\left[ \sum_{l_j=0}^\infty \frac{z_j^{i-1}}{(j+2)_{l_j}l_j!} \frac{z^{l_j}}{(\mu v)^{l_j}}
\;\;\;\;\;\sum_{l_k=0}^\infty \frac{w_k^{i-1}}{(k)_{l_k}l_k!} \frac{1}{(\mu v)^{l_k}} \right]_{\substack{i=1,2,\ldots,\alpha+2\\
 j=1,2\\
 k=3,4,\ldots,\alpha+2}}\; dz.
\end{align*}
Recalling the relations
\begin{align*}
z_j&=j+l_j, \; \; \;\;\;\;\;\;\;j=1,2\;, \\
w_j&=j+l_j-2,\; \; \; j=3,4,\ldots,\alpha+2\;,
\end{align*}
with some manipulation then gives
\begin{align}
\label{pdfasybef}
f^{(\alpha)}_V(v)& =
\frac{1}{2!3!\prod_{j=3}^{\alpha+2}(j-1)!\;}\frac{e^{-\frac{1}{\mu v}}}{\mu^4 v^5}
\int_0^1
(1-z)^{-\alpha}z^2\det\left[ \tilde{g}_{i,j}(z,v)\;\;\;\tilde{g}_{i,k-2}(1,v) \right]_{\substack{i=1,2,\ldots,\alpha+2\\
 j=1,2\\
 k=3,4,\ldots,\alpha+2}}\:dz,
\end{align}
where
\begin{align*}
\tilde{g}_{i,j}(z,v)=\sum_{p=0}^{i-1}\binom{i-1}{p}j^{i-1-p}\sum_{l_j=0}^\infty
\frac{l_j^p}{(j+2)_{l_j} l_j!}\left(\frac{z}{\mu v}\right)^{l_j}.
\end{align*}
Noting the fact that $ l_j^p=\sum_{q=0}^p \mathtt{S}^{(q)}_p
l_j(l_j-1) \cdots (l_j-q+1)$ and making the identification
\begin{equation*}
\sum_{k=0}^\infty
\frac{x^k}{(a)_k(k-N)!}=(a-1)!\;x^{\frac{N+1-a}{2}}I_{a+N-1}(2\sqrt{x})
\end{equation*}
we arrive at
\begin{align}
\label{gdef}
\tilde{g}_{i,j}(z,v)=(j+1)!\sum_{p=0}^{i-1}\sum_{q=0}^p \binom{i-1}{p}&j^{i-1-p}\mathtt{S}^{(q)}_p\left(\frac{z}{\mu
v}\right)^{\frac{q-j-1}{2}}I_{q+j+1}\left(2\sqrt{\frac{z}{\mu
v}}\right).
\end{align}
Finally, noting the fact that
$\tilde{g}_{i,j}(z,v)=(j+1)!\;g_{i,j}(z,v),\;i,j=1,2,\ldots,\alpha+2$,
and using (\ref{gdef}) in (\ref{pdfasybef}) with some algebraic
manipulation concludes the proof.
\end{proof}

For $\alpha=0$, we have the following remarkably simple closed-form
solution:
\begin{corollary}
For $\alpha=0$, (\ref{scaledpdf}) becomes
\begin{align*}
 f^{(0)}_V(v)=\frac{1}{576}\frac{e^{-\frac{1}{\mu v}}}{\mu^4 v^5}&
\left(
16\;{}_2F_3\left(3,\frac{7}{2};4,4,6;\frac{4}{\mu v}\right)-\frac{4}{\mu v}\;
{}_1F_2\left(\frac{7}{2};5,7;\frac{4}{\mu v}\right)
\right.\\
& \hspace{3cm}\left.+ \frac{3}{\mu v}\; {}_3F_4\left(\frac{7}{2},4,4;3,5,5,7;\frac{4}{\mu v}
\right)\right)\:,
\end{align*}
where ${}_pF_q\left(a_1,a_2,\ldots,a_p;b_1,b_2\ldots,b_q;z\right)$ is the generalized hypergeometric function.
\end{corollary}
This is easily obtained by setting $\alpha=0$ and integrating using
\cite[Eq. 2.15.19.1]{Prud} with some basic algebraic manipulations.

In Fig. \ref{Fig2}, we illustrate the applicability of the analytical asymptotic results in the finite context. Specifically, we compare the analytical asymptotic p.d.f. given in Theorem \ref{thm:keasy} with that of simulated points corresponding to $\alpha=0,n=10,\mu=4$ and $\alpha=1,n=50,\mu=4$. The close agreement is clearly apparent.

\begin{figure}
 \centering
 \vspace*{1.0cm}
     \subfigure[$\alpha=0$ and $n=10$]{
            \includegraphics[width=.7\textwidth]{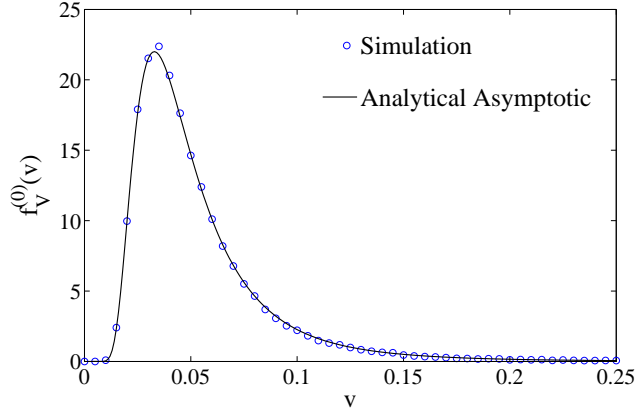}}
     \hspace{.3in}
     \subfigure[$\alpha=1$ and $n=50$]{
          \includegraphics[width=.7\textwidth]{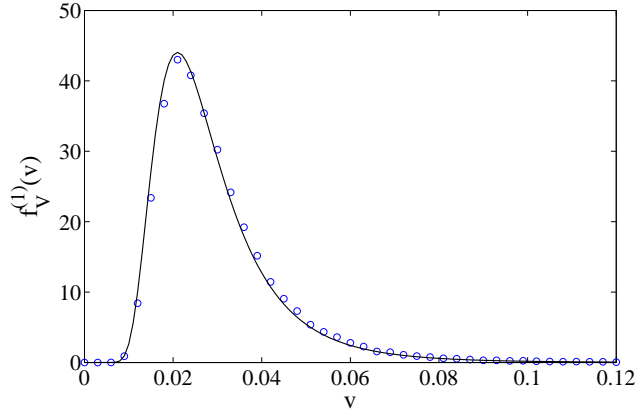}}\\
 \caption{Comparison of simulated data points and the analytical p.d.f.\ $f_V^{(\alpha)}(v)$ for two different $n,\alpha$ settings with $\mu=4$.}
 \vspace*{0.3cm}
 \label{Fig2}
\end{figure}

\appendix

\section{Some Useful Determinant Results}

The following results are useful for the proofs of Theorems
\ref{eq:ThExact} and \ref{thm:keasy}. For notational convenience we denote
$\Delta_{N}(\mathbf{x})=\prod_{1\leq l<k \leq N}\left(x_k-x_l\right)$ where $\mathbf{x}=\{x_1,x_2,\ldots,x_N\}$ with $\Delta_1(\mathbf{x})=1$.

\begin{lemma} \label{le:App}
Let $n,\alpha\in\mathbb{Z}^+$ with $n,\alpha\geq 1$. Then
\begin{equation*}
{\rm
det}\left[\prod_{i=0}^{\alpha-k-1}(\tilde{c}_l-i)\right]_{k,l=1,2,\ldots,\alpha}
=\Delta_{\alpha}(\mathbf{c})
\end{equation*}
where $\tilde{c}_l=n+\alpha-1-j_l-l$, $\mathbf{c}=\{c_1(j_1),c_2(j_2),\ldots, c_\alpha(j_\alpha)\}$ and $c_l(j_l)=l+j_l$ with $j_l=0,1,\ldots,n+\alpha-1-l$.
\end{lemma}
\begin{proof}
Let us rewrite the determinant as
\begin{align*}
\label{detexp1}
\text{det}\left[\prod_{i=0}^{\alpha-k-1}(\tilde{c}_l-i)\right]_{k,l=1,2,\ldots,\alpha}
& = \left|
\begin{array}{cccc}
\displaystyle \prod_{i=0}^{\alpha-2} (\tilde{c}_1-i) &
\displaystyle \prod_{i=0}^{\alpha-2} (\tilde{c}_2-i) &\ldots
&\displaystyle \prod_{i=0}^{\alpha-2} (\tilde{c}_\alpha-i) \\
\displaystyle \prod_{i=0}^{\alpha-3} (\tilde{c}_1-i) &
\displaystyle \prod_{i=0}^{\alpha-3} (\tilde{c}_2-i) &\ldots
&\displaystyle \prod_{i=0}^{\alpha-3} (\tilde{c}_\alpha-i)\\
\vdots & \vdots & \ddots & \vdots\\
\tilde{c}_1 & \tilde{c}_2 &\ldots
& \tilde{c}_\alpha\\
1&1&\ldots & 1
\end{array}
\right|.
\end{align*}
Now the repeated application of the following row operations
\begin{equation*}
k\text{th row}\to k\text{th row}+
\mathtt{S}^{(\alpha-k-i)}_{\alpha-k}\times (k+i)\text{th row}\;,\;\;
i=1,2,\ldots,\alpha-1-k
\end{equation*}
on each row, from $k=1$ to $k=\alpha-2$, of the above determinant
gives
\begin{align*}
& \text{det}\left[\prod_{i=0}^{\alpha-k-1}(\tilde{c}_l-i)\right]_{k,l=1,2,\ldots,\alpha}\nonumber\\
& = \left|
\begin{array}{cccc}
\displaystyle \sum_{j=0}^{\alpha-1}
\mathtt{S}^{(j)}_{\alpha-1}\prod_{i=0}^{j-1}(\tilde{c}_1-i) &
\displaystyle \sum_{j=0}^{\alpha-1}
\mathtt{S}^{(j)}_{\alpha-1}\prod_{i=0}^{j-1}(\tilde{c}_2-i) & \ldots
& \displaystyle \sum_{j=0}^{\alpha-1}
\mathtt{S}^{(j)}_{\alpha-1}\prod_{i=0}^{j-1}(\tilde{c}_\alpha-i)\\
\displaystyle \sum_{j=0}^{\alpha-2}
\mathtt{S}^{(j)}_{\alpha-2}\prod_{i=0}^{j-1}(\tilde{c}_1-i) &
\displaystyle \sum_{j=0}^{\alpha-2}
\mathtt{S}^{(j)}_{\alpha-2}\prod_{i=0}^{j-1}(\tilde{c}_2-i) & \ldots
& \displaystyle \sum_{j=0}^{\alpha-2}
\mathtt{S}^{(j)}_{\alpha-2}\prod_{i=0}^{j-1}(\tilde{c}_\alpha-i)\\
\vdots & \vdots & \ddots & \vdots\\
\tilde{c}_1 & \tilde{c}_2 & \ldots & \tilde{c}_\alpha\\
1 & 1 & \ldots & 1
\end{array}
\right|.
\end{align*}
We may use Lemma \ref{lem:stirling} to simplify each entry of the above determinant to obtain
\begin{align*}
 \text{det}\left[\prod_{i=0}^{\alpha-k-1}(\tilde{c}_l-i)\right]_{k,l=1,2,\ldots,\alpha} & =\left|
\begin{array}{cccc}
\tilde{c}_1^{\alpha-1} & \tilde{c}_2^{\alpha-1} & \ldots & \tilde{c}_\alpha^{\alpha-1}\\
\tilde{c}_1^{\alpha-2} & \tilde{c}_2^{\alpha-2} & \ldots & \tilde{c}_\alpha^{\alpha-2}\\
\vdots & \vdots & \ddots & \vdots\\
\tilde{c}_1 & \tilde{c}_2 & \ldots & \tilde{c}_\alpha\\
1 & 1 & \ldots & 1
\end{array}
\right|\nonumber\\
& =(-1)^{\lfloor \frac{\alpha}{2}\rfloor}
\text{det}\left[\tilde{c}_l^{k-1}\right]_{k,l=1,2,\ldots,\alpha},
\end{align*}
where $\lfloor\cdot\rfloor$ is the floor function. Noting the
relation
\begin{equation*}
\tilde{c}_k=n+\alpha-1-c_k(j_k)\;,\;\; k=1,2,\ldots,\alpha\;,
\end{equation*}
we obtain
\begin{align*}
\label{detfinal}
\text{det}\left[\prod_{i=0}^{\alpha-k-1}(\tilde{c}_l-i)\right]_{k,l=1,2,\ldots,\alpha}&=(-1)^{\lfloor
\frac{\alpha}{2}\rfloor}
\text{det}\left[\tilde{c}_l^{k-1}\right]_{k,l=1,2,\ldots,\alpha}\nonumber\\
& =(-1)^{\lfloor \frac{\alpha}{2}\rfloor}
\prod_{1\leq l<k\leq \alpha}\left(\tilde{c}_k-\tilde{c}_l\right)\nonumber\\
& = (-1)^{\lfloor \frac{\alpha^2}{2}\rfloor}
\prod_{1\leq l<k\leq \alpha}\left({c_k(j_k)}-{c_l(j_l)}\right)\nonumber\\
&= \Delta_{\alpha}(\mathbf{c}),
\end{align*}
where we have used the fact that $\lfloor\frac{\alpha^2}{2}\rfloor$
is an even number for $\alpha=1,2,\ldots$.
\end{proof}

\begin{lemma}\label{lem:det2}
Let $n\in \mathbb{Z}^+$ and $\alpha\in \mathbb{Z}^+\cup \{0\}$ with $n\geq 2$. Then
\begin{align*}
\det\left[\prod_{l=0}^{\alpha+1-i}(\tilde{z}_j-l)\;\;\; \;\;
 \prod_{l=0}^{\alpha+1-i}(\tilde{w}_k-l)\right]_{\substack{i=1,2,\ldots,\alpha+2\\
 j=1,2\\
 k=3,4,\ldots,\alpha+2}}
 &=\Delta_{\alpha+2}(\mathbf{w})
 \end{align*}
 where $\mathbf{w}=(z_1,z_2,w_3,\ldots,w_{\alpha+2})$ and
\begin{align*}
z_j&=j+l_j,\;\;\;\;\;\;\;\hspace{1cm} \;\tilde{z}_j=n+\alpha-z_j,\;\;\;j=1,2\;, \\
w_k&=k+l_k-2,\;\hspace{1cm} \tilde{w}_k=n+\alpha-w_k,\;k=3,4,\ldots,\alpha+2\;,
\end{align*}
with
\begin{align*}
l_j&=0,1,\ldots,n+\alpha-j,\;\;\;\;\;\;\;j=1,2\;, \\
l_k&=0,1,\ldots,n+\alpha+2-k,\;k=3,4,\ldots,\alpha+2\;.
\end{align*}
 \end{lemma}
\begin{proof}
\begin{align*}
& \det\left[\prod_{l=0}^{\alpha+1-i}(\tilde{z}_j-l)\;\;\; \;\;
 \prod_{l=0}^{\alpha+1-i}(\tilde{w}_k-l)\right]_{\substack{i=1,2,\ldots,\alpha+2\\
 j=1,2\\
 k=3,4,\ldots,\alpha+2}}\nonumber\\
 & \qquad=\left|\begin{array}{ccccc}
 \displaystyle \prod_{i=0}^{\alpha} (\tilde z_1-i) & \displaystyle \prod_{i=0}^{\alpha} (\tilde z_2-i) & \displaystyle \prod_{i=0}^{\alpha} (\tilde w_3-i)  & \ldots & \displaystyle \prod_{i=0}^{\alpha} (\tilde w_{\alpha+2}-i)\\
 \displaystyle \prod_{i=0}^{\alpha-1} (\tilde z_1-i) & \displaystyle \prod_{i=0}^{\alpha-1} (\tilde z_2-i) & \displaystyle \prod_{i=0}^{\alpha-1} (\tilde w_3-i)  & \ldots & \displaystyle \prod_{i=0}^{\alpha-1} (\tilde w_{\alpha+2}-i)\\
 \vdots & \vdots & \vdots & \ddots & \vdots\\
 \tilde z_1 & \tilde z_2 & \tilde w_3 & \ldots & \tilde w_{\alpha+2}\\
 1 & 1 & 1  & \ldots & 1
 \end{array}
 \right|\:.
\end{align*}
 Now, repeatedly
applying the following row operations
\begin{equation*}
k\text{th row}\to k\text{th row} +\mathtt{S}^{(\alpha+2-k-i)}_{\alpha+2-k}(k+i)\text{th row},\;\; i=1,2,\ldots,\alpha+1-k
\end{equation*}
on each row (i.e., $k=1,2,\ldots,\alpha$) and using Lemma \ref{lem:stirling}
gives
\begin{align*}
& \det\left[\prod_{l=0}^{\alpha+1-i}(\tilde{z}_j-l)\;\;\; \;\;
 \prod_{l=0}^{\alpha+1-i}(\tilde{w}_k-l)\right]_{\substack{i=1,2,\ldots,\alpha+2\\
 j=1,2\\
 k=3,4,\ldots,\alpha+2}}\nonumber\\
 & \hspace{3cm}=\left|
\begin{array}{ccccc}
\tilde z_1^{\alpha+1} & \tilde z_2^{\alpha+1} & \tilde w_3^{\alpha+1} & \ldots & \tilde w_{\alpha+2}^{\alpha+1}\\
\tilde z_1^{\alpha} & \tilde z_2^{\alpha} & \tilde w_3^{\alpha} & \ldots & \tilde w_{\alpha+2}^{\alpha}\\
\vdots & \vdots & \vdots & \ddots & \vdots\\
\tilde z_1 & \tilde z_2  & \tilde w_3 & \ldots & \tilde w_{\alpha+2}\\
1 & 1 & 1 & \ldots & 1
\end{array}
 \right|\nonumber\\
 & \hspace{3cm}=(-1)^{\lfloor\frac{\alpha+2}{2}\rfloor}\det\left[\tilde z_j^{i-1}\;\;\;\tilde w_k^{i-1}\right]_{\substack{i=1,2,\ldots,\alpha+2\\
 j=1,2\\
 k=3,4,\ldots,\alpha+2}}.
\end{align*}
Interestingly, the resultant simplified determinant is of Vandermonde type. This in turn gives
\begin{align*}
\label{detvand}
(-1)^{\lfloor\frac{\alpha+2}{2}\rfloor}\det\left[\tilde z_j^{i-1}\;\;\;\tilde w_k^{i-1}\right]_{\substack{i=1,2,\ldots,\alpha+2\\
 j=1,2\\
 k=3,4,\ldots,\alpha+2}}& =\det\left[ z_j^{i-1}\;\;\; w_k^{i-1}\right]_{\substack{i=1,2,\ldots,\alpha+2\nonumber\\
 j=1,2\\
 k=3,4,\ldots,\alpha+2}}= \Delta_{\alpha+2}(\mathbf{w}).
\end{align*}
 \end{proof}

\end{document}